\documentclass[reqno, 12pt, a4letter]{amsart}
\usepackage{amsmath,amsxtra,amssymb,latexsym, amscd,amsthm,pb-diagram}

\usepackage{epsfig}
\usepackage{graphics}
\usepackage{color}
\usepackage{ulem}
\usepackage[all]{xy}
\usepackage{geometry}

\numberwithin{equation}{section}
\theoremstyle{plain}
\newtheorem{thm}{Theorem}[section]

\newtheorem{lemma}[thm]{Lemma}

\newtheorem{claim}{Claim}

\theoremstyle{definition}
\newtheorem{dfn}[thm]{Definition}
\newtheorem{ex}[thm]{Example}
\theoremstyle{remark}
\newtheorem{rem}[thm]{Remark}

\newcommand{\C}{\mathbb C}
\newcommand{\N}{\mathbb N}

\newcommand{\R}{\mathbb R}

\newcommand{\ve}{\varepsilon}

\newcommand{\Ind}{\text{\rm Ind}}
\newcommand{\Sing}{\text{\rm Sing}}
\newcommand{\Int}{\text{\rm Int}}

\def\rk{\operatorname{rank}}
\def\dim{\operatorname{dim}}
\def\grad{\operatorname{grad}}

\def\ees{{\accent"5E e}\kern-.385em\raise.2ex\hbox{\char'23}\kern-.08em}
\def\EES{{\accent"5E E}\kern-.5em\raise.8ex\hbox{\char'23 }}
\def\ow{o\kern-.42em\raise.82ex\hbox{
\vrule width .12em height .0ex depth .075ex \kern-0.16em \char'56}\kern-.07em}
\def\OW{O\kern-.460em\raise1.36ex\hbox{
\vrule width .13em height .0ex depth .075ex \kern-0.16em \char'56}\kern-.07em}

\begin{document}
\title[Atypical values at infinity of real polynomial maps]{Atypical values at infinity of real polynomial maps with $2$-dimensional fibers}

\author[Masaharu Ishikawa]{Masaharu Ishikawa}
\address{Faculty of Economics, Keio University, 4-1-1, Hiyoshi, Kouhoku, Yokohama, Kanagawa 223-8521, Japan}
\email{ishikawa@keio.jp}

\author[Tat-Thang Nguyen]{Tat-Thang Nguyen}
\address{Institute of Mathematics, Vietnam Academy of Science and Technology, 18 Hoang Quoc Viet road, Cau Giay district, 11300 Hanoi, Vietnam}
\email{ntthang@math.ac.vn}

\thanks{
The second author thanks Keio University for warm hospitality during his visit.
The first author is supported by JSPS KAKENHI Grant numbers JP19K03499, JP23K03098, 
JP23H00081,
JSPS-VAST Joint Research Program, Grant number JPJSBP120219602,
and Keio University Academic Development Funds for Individual Research.
The second author is partially supported by Vietnam National Foundation for Science and Technology Development (NAFOSTED) under grant number 101.04-2023.33.}

\begin{abstract}
We characterize atypical values at infinity of a real polynomial function of three variables by a certain sum of indices of the gradient vector field of the function restricted to a sphere with a sufficiently large radius.
This is an analogy of a result of Coste and de la Puente for real polynomial functions with two variables. 
We also give a characterization of atypical values at infinity of a real polynomial map whose regular fibers are $2$-dimensional surfaces.
\end{abstract}

\maketitle

\section{Introduction}

Let $f:\R^n\to \R^m$ be a real polynomial map, 
$\Sing(f)$ be the set of singular points of $f$ in $\R^n$, and $K_0(f)=f(\Sing(f))$.
A {\it bifurcation set} of $f$ is the smallest set of values in $\R^m$ outside which $f$ is a locally trivial fibration.
This is a semialgebraic set of  codimension at least one~\cite{Tho69, Ver76, KOS00}.
A regular value $\lambda\in f(\R^n)\setminus K_0(f)$ is called a {\it typical value at $\infty$} of $f$ if 
there is an open neighborhood over which $f$ is a trivial fibration. 
Otherwise, $\lambda$ is called an {\it atypical value at $\infty$} of $f$.
For example, the polynomial map $f(x,y)=x(xy+1)$ has no critical value but its bifurcation set is $\{0\}$.
There are several studies about the bifurcation sets of real polynomial maps, see for instance~\cite{tz, cp, JT17, INP19, DJT21}. 

Suppose $m=1$, that is, $f:\R^n\to\R$ is a real polynomial function.
Let $B_{a,R}$ be the closed ball in $\R^n$ centered at a point $a\in\R^n$ and of radius $R>0$.
Set
\[
   \Gamma=\left\{x\in\R^n\mid   
   \rk \begin{pmatrix}
   x-a \\
 \grad f
\end{pmatrix}\leq 1\right\}.
\]
Note that $\Sing(f)\subset \Gamma$.
We choose a center $a\in\R^n$ and a sufficiently large $R>0$ so that $\Gamma$ is transverse to $\partial B_{a,r}$ for any $r>R$ and 
$\Gamma\setminus \Int B_{a,R}$ is homeomorphic to $\Gamma\cap \partial B_{a,R}\times [0,1)$. 
Each connected component of $\Gamma\setminus\Int B_{a,R}$ is contained in either $\Sing(f)$ or $\Gamma\setminus \Sing(f)$. 
Throughout the paper, we always choose the center $a$ generic so that each
connected component of $\Gamma\setminus (\Sing(f)\cup\Int B_{a,R})$ is a curve.
These curves are called {\it tangency branches at $\infty$} of $f$.

For each point $p\in (\Gamma\setminus \Sing(f))\cap \partial B_{a,R}$, let $\Gamma_p$ denote the tangency branch at $\infty$ of $f$ passing through $p$.
Set $x_p(r)=\Gamma_p\cap \partial B_{a,r}$ for $r\geq R$
and define 
\[
  \lambda_p=\lim_{r\to\infty} f(x_p(r))\in \R\cup\{\pm\infty\}.
\]
Let $T_\infty(f)$ denote the set of values $\lambda\in\R$ for which there exists a curve $x:[R,\infty)\to\Gamma$ with $x(r)\in\Gamma\cap\partial B_{a,r}$ and $\lim_{r\to\infty}f(x(r))=\lambda$. 
Note that
\[
   T_\infty(f)\subset\{\lambda_p\in\R\mid p\in (\Gamma\setminus \Sing(f))\cap\partial B_{a,R}\}\cup K_0(f).
\]

The aim of this paper is to characterize atypical values of $f$ by observing its behavior on the sphere $\partial B_{a,R}$ with a sufficiently large radius $R>0$. 
Specifically, we focus on the vector field $X_{a,R}$ on $\partial B_{a,R}$ defined by the gradient vector field of the restriction of $f$ to $\partial B_{a,R}$.
For each isolated zero $p$ of $X_{a,R}$, the index $\Ind_p(X_{a,R})$ is defined by the degree of the map from $\partial B_{p,\ve}$ to the $(n-1)$-dimensional sphere given by $x\mapsto \frac{X_{a,R}(x)}{\|X_{a,R}(x)\|}$, where $\ve>0$ is a sufficiently small real number.
Note that each point of $(\Gamma\setminus \Sing(f))\cap\partial B_{a,R}$ is an isolated zero of $X_{a,R}$.
For each $\lambda\in T_{\infty}(f)$, let $\Gamma^{(\lambda)}$ be the union of tangency branches $\Gamma_p$ with $\lambda_p=\lambda$.  
For each connected component $\Omega$ of $\partial B_{a,R}\setminus f^{-1}(\lambda)$,
set 
\[
   \Ind(\lambda,\Omega)=\sum_{p\in\Gamma^{(\lambda)}\cap\Omega} \Ind_p(X_{a,R}).
\]


We focus on the case $n=3$. In this case, since regular fibers of $f$ are of dimension~$2$, their topology can be determined by the indices of the vector field $X_{a,R}$.
In consequence, we obtain the following theorem.
For the definition of a vanishing component, see Section~\ref{sec21}.

\begin{thm}\label{thm1}
Let $f:\R^3\to\R$ be a polynomial function and $\lambda\in T_\infty(f)\setminus K_0(f)$.
If $\Ind(\lambda,\Omega)\neq 0$ for some connected component $\Omega$ of $\partial B_{a,R}\setminus f^{-1}(\lambda)$ then $\lambda$ is an atypical value at $\infty$ of $f$. 
Conversely, if there does not exist a vanishing component at $\infty$ when $t$ tends to $\lambda$ and $\Ind(\lambda, \Omega)=0$ for any connected component $\Omega$ of $\partial B_R\setminus f^{-1}(\lambda)$ then $\lambda$ is a typical value at $\infty$ of $f$.
\end{thm}

In the proof, it is shown that if $\Ind(\lambda,\Omega)\neq 0$ for some $\Omega$ then,
for $t$ sufficiently close to $\lambda$, there exists a connected component of $f^{-1}(t)\setminus \Int B_{a,R}$ diffeomorphic to a disk. 
This interpretation can be used when we generalize the assertion in Theorem~\ref{thm1} to polynomial maps $F:\R^n\to \R^{n-2}$ for $n\geq 3$. 
The statement is the following.

\begin{thm}\label{thm2}
Let $F:\R^n\to \R^{n-2}$ be a polynomial map, where $n\geq 3$.
Suppose that the radius $R>0$ of $B_{a,R}$ is sufficiently large.
Then, $\lambda\in F(\R^n)\setminus K_0(F)$ is a typical value at $\infty$ of $F$ if and only if the following are satisfied:
\begin{itemize}
\item[(1)] There is no vanishing component at $\infty$ when $t$ tends to $\lambda$;
\item[(2)] There exists a neighborhood $D$ of $\lambda$ in $\R^{n-2}$ such that, for any $t\in D$, 
\begin{itemize}
\item[(2-1)] $F^{-1}(t)\setminus \Int B_R$ has no compact, connected component, and
\item[(2-2)] $\chi(F^{-1}(t))= \chi(F^{-1}(\lambda))$ holds.
\end{itemize}
\end{itemize}
\end{thm}


The above theorem is stated again in Section~\ref{sec5} (Theorem~\ref{thm51}), where a precise condition for the radius $R$ is given. The condition~(2-1) is added instead of the condition about the indices in Theorem~\ref{thm1}. 
Note that atypical values of an algebraic family of real curves,
which can be seen as a restriction of a polynomial map from $\R^n$ to $\R^{n-1}$,
are characterized by the conditions~(1) and (2-2)~\cite{tz}. See also~\cite{JT17}.
Atypical values of a holomorphic map between connected complex manifolds $M\to B$ with
$\dim_\C M=\dim_\C B+1$ are also characterized by the conditions~(1) and (2-2)~\cite{JT18}.

This paper is organized as follows. In Section~\ref{sec2}, we prove a few lemmas concerning a choice of the center $a$ and the radius $R$ of the ball $B_{a,R}$.
In Section~\ref{sec25}, two examples of polynomial functions $f:\R^3\to\R$, which are
based on examples in~\cite{tz} (also~\cite{cp}), are given.
In Section~\ref{sec3}, we prove a theorem that characterizes a vanishing component at infinity of a real polynomial function. Using this theorem, we can obtain some argument for detecting a vanishing component
at infinity, see Remark~\ref{rem32}. Section~\ref{sec4} is devoted to the proof of Theorem~\ref{thm1},
and Section~\ref{sec5} is devoted to the proof of Theorem~\ref{thm2}.

\section{Preliminaries}\label{sec2}

\subsection{Vanishing component}\label{sec21}

In this section we give the definition of a vanishing component at $\infty$ for a polynomial map from $\R^n$ to $\R^m$ with $n>m\geq 1$.

\begin{dfn}\label{dfn21}
Let $F:\R^n\to\R^m$ be a polynomial map.
It is said that {\it there is a vanishing component at $\infty$ when $t$ tends to $\lambda$} if there exists a sequence of points $\{t_k\}$ in $\R^m$ such that
\[
\lim_{k\to\infty} t_k=\lambda
\quad\text{and}\quad
\lim_{k\to\infty}\max_{i}\inf\{\|x\|\in\R \mid x\in Y_{t_k,i}\}=\infty,
\]
where $Y_{t,1},\ldots,Y_{t,n_t}$ are the connected components of $F^{-1}(t)$.
\end{dfn}

\begin{rem}
The existence of a vanishing component at $\infty$ does not change even if the distance function $\|x\|$ is replaced by $\|x-a\|$ for any point $a\in \R^n$.
\end{rem}

\subsection{The center of the ball $B_{a,R}$}\label{sec22}

Let $f:\R^n\to\R$ be a polynomial function,
$\Sing(f)$ be the set of critical points of $f$ in $\R^n$, and $K_0(f)=f(\Sing(f))$.

\begin{lemma}\label{lemma23}
Let $f: \mathbb{R}^n\to \R$ be a polynomial function, 
$K$ be a finite set in $f(\R^n)\setminus K_0(f)$, and
$A_K$ be the set of points $a$ in $\mathbb{R}^n$ satisfying that,
for each $\lambda\in K$, 
there exists an open interval $I_\lambda$ in $\R$ containing $\lambda$ such that
the function on $f^{-1}(t)$ defined by $x\mapsto \|x-a\|^2$ has only non-degenerate critical points for any $t\in I_\lambda\setminus\{\lambda\}$.
Then the set $A_K$ is dense in $\R^n$.
\end{lemma}

\begin{proof}
Set
\[
   S=\left\{(x,v,t)\in \mathbb{R}^n\times \mathbb{R}^n\times (\mathbb{R}\setminus K_0(f)) \mid f(x)=t, 
   \rk 
   \begin{pmatrix}
   v\\
   \grad f
\end{pmatrix}\leq 1\right\}.
\]
It is easy to check that $S$ is a semialgebraic set of dimensional $n+1$ having no singular points.

Consider the ``endpoint'' map (see \cite{Mil63}):
\[
   E: S\to \mathbb{R}^n\times (\mathbb{R}\setminus K_0(f)),\quad (x,v,t)\mapsto (x+v, t).
\]
By the Sard Theorem, the set $E(\Sing(E))$ of singular values of $E$ has measure $0$. We can also check that $E(\Sing(E))$ is a semialgebraic set in $\mathbb{R}^n\times (\mathbb{R}\setminus K_0(f))$ of dimension at most $n$.
By \cite[Lemma 6.5]{Mil63}, $(a,t)\in E(\Sing(E))$ if and only if 
the function on $f^{-1}(t)$ defined by $x\mapsto \|x-a\|^2$ has a degenerate critical point.

We will prove the following claim: For each point $a\in \mathbb{R}^n$, any neighborhood of $a$ in $\mathbb{R}^n$ contains at least one point $x\neq a$ such that the intersection $(\{x\}\times \mathbb{R})\cap E(\Sing(E))$ is an isolated set. 
This implies that $A_K$ is dense in $\mathbb{R}^n$.

For a contradiction, we assume that there exist a point $a\in \mathbb{R}^n$ and a small neighborhood $U$ of $a$ in $\mathbb{R}^n$ satisfying that, for each $x\in U\setminus \{a\}$, there is an open interval $I_x\subset \mathbb{R}$ such that $\{x\}\times I_x\subset E(\Sing(E))$.

Since $E(\Sing(E))$ is a semialgebraic set in $E(S)$ of codimension at least one,
its Zariski closure $V$ in $\mathbb{R}^n\times \mathbb{R}$ is an algebraic subset of dimension at most $n$.
Let $\pi:V\to \mathbb{R}^n$ be the projection from $V\subset\R^n\times\R$ to $\R^n$ defined by $(x,t)\mapsto x$.
Since $\{x\}\times I_x\subset V$ for $x\in U\setminus\{a\}$, the inclusion $U\setminus \{a\}\subset \pi(V)$ holds.

On the other hand, it implies from~\cite{Tho69, Ver76}
that there exists an open ball $B\subset U\setminus\{a\}$ such that $\pi$ is trivial on $B$, which means that $\pi^{-1}(B)\subset V$ is diffeomorphic to $B\times \pi^{-1}(x) $ for $x\in B$. From the inclusion $U\setminus \{a\}\subset \pi(V)$, we get $\{x\}\times I_x\subset \pi^{-1}(x)$. Therefore $\dim \pi^{-1}(B)=\dim U+1=n+1$. This contradicts $\dim V\leq n$.
\end{proof}

\begin{rem}\label{remark24}
In Lemma~\ref{lemma23}, a point in $f^{-1}(t)$ around which the function on $f^{-1}(t)$ defined by $x\mapsto\|x-a\|^2$ is locally constant is regarded as a degenerate critical point.
\end{rem}

\subsection{Topology of fibers and indices of vector fields on the sphere}\label{sec23}

Choose a center $a\in\R^n$ of $B_{a,R}$ generic and the radius $R>0$ sufficiently large.
The interior of $B_{a,R}$ is denoted by $\Int B_{a,R}$ and its boundary is by $\partial B_{a,R}$.
For each point $p\in (\Gamma\setminus \Sing(f))\cap \partial B_{a,R}$, let $\Gamma_p$ denote the tangency branch at $\infty$ of $f$ passing through $p$.
Set $x_p(r)=\Gamma_p\cap\partial B_{a,r}$, then $f(x_p(r))$ is monotone with respect to the parameter $r$.
We use the following notations:
\begin{itemize}
 \item $f\nearrow \lambda$ along $\Gamma_p$ means that $f(x_p(r))$ is monotone increasing for $r\geq R$ and $\lim_{r\to\infty}f(x_p(r))=\lambda$.
 \item $f\searrow \lambda$ along $\Gamma_p$ means that $f(x_p(r))$ is monotone decreasing for $r\geq R$  and $\lim_{r\to\infty}f(x_p(r))=\lambda$.
\end{itemize}

\begin{rem}\label{rem25}
Let $\Gamma_p$ be the tangency branch at $\infty$ of $f$ passing through $p\in(\Gamma\setminus\Sing(f))\cap \partial B_{a,R}$. We have the following remarks.
\begin{itemize}
\item[(1)] The point $p$ is a critical point of the following two functions:
\[
\begin{split}
   & f|_{\partial B_{a,r_a(p)}}: \partial B_{a,r_a(p)}\to \R, \text{\,where $r_a(p)=\|p-a\|$},\\
   & r_a|_{f^{-1}(f(p))}: f^{-1}(f(p))\to \R, \text{\,where $r_a(x)=\|x-a\|$}.
\end{split}
\]
\item[(2)] Suppose that $f\nearrow \lambda_p$ along $\Gamma_p$. Then, $p$ is a local maximum (resp. minimum) point of $f|_{\partial B_{a,r_a(p)}}$ if and only if it is a local minimum (resp. maximum) point of $r_a|_{f^{-1}(f(p))}$.
\item[(3)] Suppose that $f\searrow \lambda_p$ along $\Gamma_p$. Then, $p$ is a local maximum (resp. minimum) point of $f|_{\partial B_{a,r_a(p)}}$ if and only if it is a local maximum (resp. minimum) point of $r_a|_{f^{-1}(f(p))}$ (cf.~Example~\ref{ex27}).
\end{itemize}

For simplicity,  we denote by $\mathcal{P}$ one of the properties ``local maximum'', ``local minimum'', 
``neither local maximum nor local minimum''. 
\end{rem}

\begin{lemma}\label{lemma25}
There exists a sufficiently large radius $R>0$ such that, for each $p\in (\Gamma\setminus \Sing(f))\cap \partial B_{a,R}$, the property $\mathcal P$ of $f|_{\partial B_{a,r}}$ is constant on $\Gamma_p$.
\end{lemma}

\begin{proof}
For each property $\mathcal P$, define the subset $V_{\mathcal{P}}$ of $\mathbb{R}^n$ by
\[
   V_{\mathcal{P}}= \{x\in \mathbb{R}^n \mid \text{$x$ is a $\mathcal{P}$ point of $f|_{\partial B_{a,r_a(x)}}$ for $r_a(x)=\|x-a\|\geq R$}\}.
\]
We will show that, for each $\mathcal{P}$, the set $V_{\mathcal{P}}$ is a semi-algebraic set. 
If $\mathcal P$ is local maximum, the set $V_{\mathcal{P}}$ is represented in terms of the first-order formulas as follows (for the definitions of first-order formulas, see \cite{BCR,HP}):
\[
   V_{\mathcal P}= \{x\in \R^n \mid \exists\ve\in \R\, ((y\in \R^n, \|y\|=\|x\|, \|y-x\|<\ve) \Rightarrow f(y)\leq f(x))\}.
\]
Hence, it implies from the Tarski-Seidenberg Theorem (see \cite[Proposition 2.2.4]{BCR} or \cite[Theorem 1.6]{HP}) that $V_{\mathcal P}$ is a semialgebraic set.  
The set $V_{\mathcal P}$ for $\mathcal P$ being local minimum is also semialgebraic by a similar argument.
If $\mathcal P$ is neither local maximum nor local minimum, then the set $V_{\mathcal P}$ is the complement of the above two semialgebraic sets. Therefore it is also semialgebraic.

Now, $\Gamma_p\cap V_{\mathcal{P}}$ is a semialgebraic subset of a curve for each $\mathcal P$.
Hence we can choose $R>0$ sufficiently large so that 
each $\Gamma_p$ is contained in some of $V_{\mathcal{P}}$.
\end{proof}

\subsection{Choice of the radius $R$}\label{sec24}

Define the set $K_\infty(f)$ by
\[
\begin{split}
 K_\infty(f)=\{t\in\R\mid \; & \text{there exists a sequence $\{x_k\}$ in $\R^n$ such that $\|x_k\|\to\infty$, }  \\
 &\text{$f(x_k)\to t$, and $\|x_k\|\,\|\grad f(x_k)\|\to 0$ as $k\to\infty$}\}.
\end{split}
\]
Note that $K_\infty(f)$ is a finite set and satisfies $T_\infty(f)\subset K_0(f)\cup K_\infty(f)$.
We choose a generic point $a\in \mathbb{R}^n$ as in Lemma \ref{lemma23} with respect to the set $K=K_{\infty}(f)\setminus K_0(f)$.
Choose an open interval $I_{\lambda}$ for each $\lambda\in K$ so that
$I_{\lambda}\cap I_{\lambda'}=\emptyset$ for $\lambda\ne \lambda' \in T_\infty(f)$.
We choose the radius $R>0$ sufficiently large so that the following properties hold:
\begin{itemize}
\item[(i)] $\Gamma\setminus \Int B_{a,R}$ is homeomorphic to $(\Gamma\cap \partial B_{a,R})\times [0,1)$ and, for each $p\in (\Gamma\setminus\Sing(f))\cap \partial B_{a,R}$,
\[
  \Gamma_p\cap \bigcup_{\lambda\in T_\infty(f)}f^{-1}(\lambda)=\emptyset.
\]
\item[(ii)] $R>0$ satisfies the condition in Lemma~\ref{lemma25}.
Since the center $a$ is chosen as in Lemma~\ref{lemma23}, 
the property ``neither local maximum nor local minimum'' for tangency branches is replaced by ``saddle''.
\item[(iii)] For each $\lambda \in K$, each connected component $Y$ of $f^{-1}(\lambda)\setminus \Int B_{a,r}$ intersects $\partial B_{a,r}$ transversely for any $r\geq R$. In particular, $Y$ is diffeomorphic to $(Y\cap \partial B_{a,r})\times [0,1)$ for any $r\geq R$.
\item[(iv)] $\{f(x)\mid x\in \Gamma_p\} \subset I_{\lambda_p}$ holds for any $p\in(\Gamma\setminus \Sing(f))\cap\partial B_{a,R}$.
\end{itemize}

In the following sections, we always assume that the radius $R>0$ is sufficiently large so that these properties hold.

\subsection{Examples}\label{sec25}

We give two examples of polynomial functions $f:\R^3\to\R$ of the form $f(x,y,z)=g(x,y)$, where $g(x,y)$ is a polynomial function of two variables.

\begin{ex}\label{ex27}
Let $g:\R^2\to\R$ be the following polynomial function:
\[
g(x,y)=2y^5+4xy^4+(2x^2-9)y^3-9xy^2+12y.
\]
This example is given in~\cite[Example 3.4]{tz}. The shapes of fibers around the infinity is studied in~\cite{cp} explicitly, which is given as in Figure~\ref{fig1}. There are eight tangency branches, four of which are on the right-hand side and the other four are on the left-hand side. 
The arrow on each tangency branch represents the direction in which the value of $f$ increases.
For example, for the right-top tangency branch $\Gamma_{p_1}$, we have
$g\searrow 0$ along $\Gamma_{p_1}$, $p_1$ is a local minimum of $g|_{B_{a,r_a(p_1)}}$, and it is a local minimum of $r_a|_{g^{-1}(g(p_1)}$, where $r_a(x)=\|x-a\|$. This function has no vanishing component at $\infty$.

\begin{figure}[htbp]
\includegraphics[scale=0.9, bb=130 524 544 712]{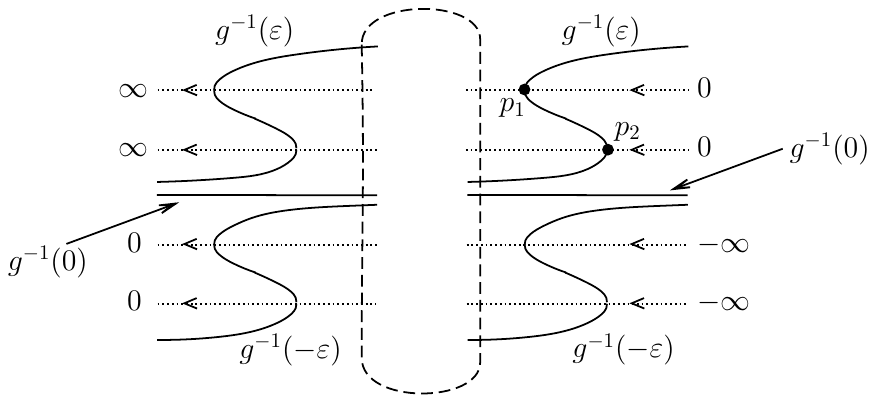}
\caption{Fibers around the infinity in Example~\ref{ex27}. The oriented dotted curves are tangency branches at $\infty$.}
\label{fig1}
\end{figure}

Let $f:\R^3\to\R$ be a polynomial function given by $f(x,y,z)=g(x,y)$. 
In~\cite[Example 3.4]{tz}, the function $g$ is obtained from $h(x,y)=y(2x^2y^2-9xy+12)$ 
as $g(x,y)=h(x+y,y)$. From this form, we can see that $g^{-1}(0)\cap \partial B_{a,R}$ is given by $\{y=0\}\cap \partial B_{a,R}$, which is a connected, simple closed curve on the $2$-sphere $\partial B_{a,R}$. 
The complement of this curve in $\partial B_{a,R}$ consists of two open disks. 
We denote the one where $y$ is positive by $\Omega_1$ and the other by $\Omega_2$.
By choosing the center $a$ of $B_{a,R}$ on $z=0$, we may assume that all tangency branches in Figure~\ref{fig1} are on the plane $z=0$.
Then, for example, the point $p_1$ is local minimum of $f|_{B_{a,r_a(p_1)}}$ and also local minimum of $r_a|_{f^{-1}(f(p_1)}$, where $r_a(x)=\|x-a\|$. This is in the case~(3) of Remark~\ref{rem25}. The index is $\Ind_{p_1}(X_{a,R})=1$. 
On the other hand, the singularity of $r_a|_{f^{-1}(f(p_2)}$ on the tangency branch $\Gamma_{p_2}$ passing through the point $p_2$ in the figure becomes a saddle, and therefore its index is $\Ind_{p_2}(X_{a,R})=-1$. 
The union $\Gamma^{(0)}$ of tangency branches at $\infty$ of $f$ along which either $f\searrow 0$ or $f\nearrow 0$
has
no other tangency branch passing through the region $\Omega_1$.
Hence we have
\[
\begin{split}
   \Ind(\lambda,\Omega_1)&=\sum_{p\in\Gamma^{(0)}\cap\Omega_1} \Ind_p(X_{a,R}) \\
   &=\Ind_{p_1}(X_{a,R})+\Ind_{p_2}(X_{a,R}) \\
   &=1+(-1)=0.
\end{split}
\]
By the same observation, we have $\Ind(\lambda,\Omega_2)=0$. 
Then, by Theorem~\ref{thm1}, we can conclude that $0$ is a typical value at $\infty$ of $f$.
\end{ex}

\begin{ex}\label{ex28}
Let $g:\R^2\to\R$ be the following polynomial function:
\[
g(x,y)=x^2y^3(y^2-25)^2+2xy(y^2-25)(y+25)-y^4-y^3+50y^2+51y-575.
\]
This example is given in~\cite[Example 3.1]{tz}. The shapes of fibers around the infinity is studied in~\cite{cp} explicitly after replacing $x$ by $x+y$ to avoid vertical tangency at infinity. The fibers are given as in Figure~\ref{fig2}.
There are two component vanishing at $\infty$ when $t$ tends to $0$. 
The word ``cleaving'' means that the point on the tangency branch goes to $\infty$ when $t$ tends to $0$,
so that the curve cleaves locally into two curves. There are two cleaving curves.

\begin{figure}[htbp]
\includegraphics[scale=0.9, bb=130 479 587 712]{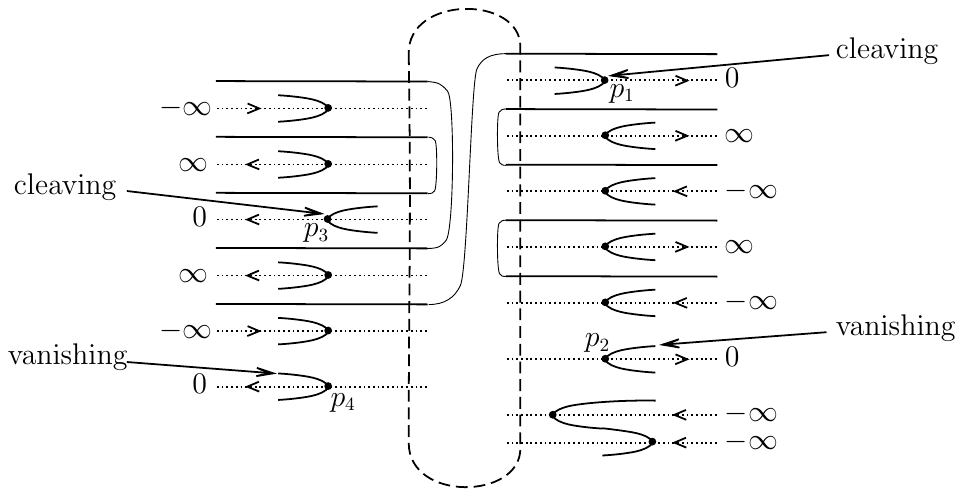}
\caption{Fibers around the infinity in Example~\ref{ex28}. The oriented dotted curves are tangency branches at $\infty$. All horizontal solid lines are curves representing $g^{-1}(0)$.}
\label{fig2}
\end{figure}

Let $f:\R^3\to\R$ be a polynomial function given by $f(x,y,z)=g(x,y)$. 
Using Mathematica, we can see that the curve of $g^{-1}(0)$ inside the dotted circle is as shown in Figure~\ref{fig2}.
Note that it is explained in~\cite[Example 3.1]{tz} that
if $|t|$ is sufficiently small then $g^{-1}(t)$ is a disjoint union of
five non-compact connected components.
Thus, the curves $f^{-1}(0)\cap\partial B_{a,R}$ on the sphere $\partial B_{a,R}$ becomes as shown in Figure~\ref{fig2-2}. There are
five
circles.
Let $X_{a,R}$ be the gradient vector field of $f|_{\partial B_{a,R}}$. 
We have $\Ind_{p_1}(X_{a,R})=\Ind_{p_3}(X_{a,R})=-1$ and $\Ind_{p_2}(X_{a,R})=\Ind_{p_4}(X_{a,R})=1$.
On the region $\Omega_1$ depicted in the figure, we have
\[
\begin{split}
   \Ind(0,\Omega_1)&=\Ind_{p_1}(X_{a,R})+\Ind_{p_2}(X_{a,R})+\Ind_{p_4}(X_{a,R}) \\
   &=(-1)+1+1=1\ne 0.
\end{split}
\]
Hence $0$ is an atypical value at $\infty$ of $f$ by Theorem~\ref{thm1}.
We can get the same conclusion from the region $\Omega_2$ depicted in the figure 
since $\Ind(0,\Omega_2)=\Ind_{p_3}(X_{a,R})=-1\ne 0$.
\begin{figure}[htbp]
\includegraphics[scale=0.9, bb=180 580 363 711]{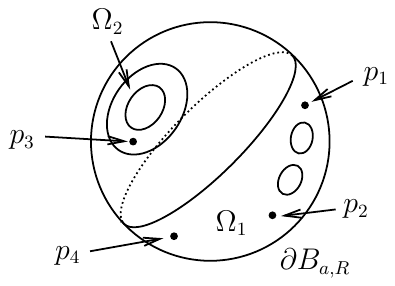}
\caption{The curves $f^{-1}(0)\cap\partial B_{a,R}$ on the sphere $\partial B_{a,R}$.}
\label{fig2-2}
\end{figure}
\end{ex}



\section{A characterization of vanishing component at infinity}\label{sec3}

Let $f:\R^n\to\R$ be a polynomial function.
Hereafter we omit $a$ in the suffix of $B_{a,r}$ for $r>0$ and denote it by $B_r$ for simplicity.
Each critical point $p\in \partial B_R$ of $f|_{\partial B_R}$ not lying on $\Sing(f)$ is a point in $(\Gamma\setminus \Sing(f))\cap \partial B_R$. Hence it has a tangency branch $\Gamma_p$.

\begin{thm}\label{thm31}
Suppose $n\geq 2$ and $\lambda\in T_{\infty}(f)\setminus K_0(f)$. 
There is a vanishing component at $\infty$ when $t$ tends to $\lambda$ with $t>\lambda$  (resp. $t<\lambda$) if and only if there exists a local minimum (resp. maximum) point $p\in \partial B_R$ of $f|_{\partial B_R}$ with $f\searrow \lambda$ (resp. $f\nearrow \lambda$) along $\Gamma_p$ such that the intersection of the connected component of $f^{-1}(f(p))$ containing $p$ with $\partial B_R$ consists of isolated points.
\end{thm}

\begin{proof}
We first prove the ``only if'' assertion. We only prove the assertion in the case where $t$ tends to $\lambda$ with $t>\lambda$.
The proof for the other case is similar.
Let $\{Y_t\}$ be a continuous family of connected components of $f^{-1}(t)$ that vanishes at $\infty$ when $t$ tends to $\lambda$.
Tangency branches intersecting $\{Y_t\}$ are contained in a connected component $H$ of 
$\R^n\setminus (f^{-1}(\lambda)\cup\text{Int} B_R)$ by the property~(i)
about the choice of the radius $R$ in Section~\ref{sec24}.
Remark that $H$ is possibly $\R^n\setminus \text{Int} B_R$.
Set $\Omega=H\cap \partial B_R$.
Either $\Omega=\partial B_R$, or $\Omega\subset \partial B_R$ is bounded by a finite number of circles belonging to $f^{-1}(\lambda)\cap \partial B_R$.
Since $t$ tends to $\lambda$ with $t>\lambda$, we have $\Omega\subset \{x\in\partial B_R\mid f(x)>\lambda\}$. 
Let $S_\Omega$ denote the set of local minimum points of $f|_{\partial B_R}$ in $\Omega\cap\Gamma^{(\lambda)}$, where $\Gamma^{(\lambda)}$ is the union of tangency branches at $\infty$ of $f$ along which either $f\nearrow\lambda$ or $f\searrow\lambda$. 
By the definition of a vanishing component at $\infty$ in Section~\ref{sec21} and Remark~\ref{rem25}~(1),
the function $r_a(x)=\|x-a\|$ restricted to $f^{-1}(t)$ has a local minimum point $y$ on $\Omega\cap \Gamma^{(\lambda)}$. Since $t>\lambda$, it satisfies that $f\searrow \lambda$ along $\Gamma_y$.
Hence, by Remark~\ref{rem25}~(3), $y$ is a local minimum point of $f|_{\partial B_R}$, that is, $y$ is a point in $S_\Omega$. In particular,  $S_\Omega$ is non-empty.

Set $\delta=\min_{x\in S_\Omega} f(x)$
and let $p$ be a point in $S_\Omega$ such that $f(p)=\delta$ and $f\searrow \lambda$ along $\Gamma_p$. Let $(\lambda,\delta]$ be the range of the parameter $t$ of $Y_t$. 
We will show that $Y_\delta\cap \partial B_R$ consists of isolated points. 

Assume for a contradiction that $Y_\delta\cap \partial B_R$ is not isolated. 

\begin{claim}\label{claim0-2}
$Y_\delta\cap \Omega$ is not isolated. 
\end{claim}

\begin{proof}
Assume that $Y_\delta\cap \Omega$ is isolated. Then, all points in $Y_\delta\cap\Omega$ are local minima of $r_a|_{Y_\delta}:Y_\delta\to\R$, where $r_a(x)=\|x-a\|$. The inequality $\lambda<\delta$ implies that 
$Y_\delta\cap f^{-1}(\lambda)=\emptyset$. Hence $Y_\delta\subset H$, see Figure~\ref{fig3}. This inclusion implies
$Y_\delta\cap \partial B_R=Y_\delta\cap \Omega$. 
However, the right-hand side is isolated while the left-hand is not. This is a contradiction.
\end{proof}

\begin{figure}[htbp]
\includegraphics[scale=1, bb=148 622 389 712]{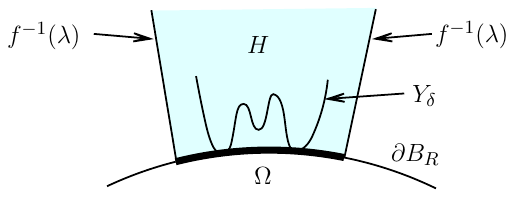}
\caption{$Y_\delta\cap f^{-1}(\lambda)=\emptyset$ implies $Y_\delta\subset H$.}
\label{fig3}
\end{figure}

We continue the proof of Theorem~\ref{thm31}. 
Let $Y_{[\lambda,\delta]}$ be the connected component of $f^{-1}([\lambda, \delta])$ containing $p$. 
There are two cases:

Case 1: $Y_{[\lambda,\delta]}\cap f^{-1}(\lambda)=\emptyset$ (cf.~Figure~\ref{fig4}). 
Set $\Omega_{[\lambda,\delta]}=Y_{[\lambda,\delta]}\cap \bar \Omega$,
where $\bar\Omega$ is the closure of $\Omega$ in $\partial B_R$. 
Since $Y_\delta\cap \Omega$ is not isolated by Claim~\ref{claim0-2}, 
$Y_\delta\cap \Omega$ has a connected component $C$ of dimension at least $1$.
A point in $\Omega$ at which $Y_\delta$ is tangent to $\Omega$ belongs to a tangency branch at $\infty$ of $f$ and hence it is isolated in $\Omega$. In particular, it cannot be in $C$. 
This means that $Y_\delta$ and $\Omega$ intersect along $C$ transversely. 
Therefore, since $f$ is continuous on $\bar\Omega$, $Y_{[\lambda,\delta]}\cap \Omega$ has a connected component $C$ of dimension $n-1\geq 1$. 
This set $C$ is a compact subset of $\Omega$.
Due to a generic choice of the center $a$ in Lemma~\ref{lemma23},
the restriction of $f$ to $C$ cannot be a constant function.
Hence $f$ is not a constant function on $\Omega_{[\lambda,\delta]}$.
Since $\partial\bar\Omega\subset f^{-1}(\lambda)$ (possibly $\partial\bar\Omega=\emptyset$) and 
$Y_{[\lambda,\delta]}\cap f^{-1}(\lambda)=\emptyset$,
we have $\partial \Omega_{[\lambda,\delta]}\subset f^{-1}(\delta)$ (possibly $\partial \Omega_{[\lambda,\delta]}=\emptyset$).
Hence, there exists a local minimum point $q$ of $f|_{\partial B_R}$ in the interior of $\Omega_{[\lambda,\delta]}$ with $\lambda<f(q)<\delta$.

\begin{figure}[htbp]
\includegraphics[scale=1.16, bb=151 588 392 710]{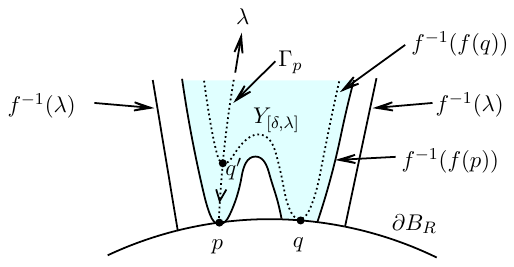}
\caption{A schematic picture for the proof in Case 1.}
\label{fig4}
\end{figure}

Since $\lambda<f(q)<\delta=f(p)$, there exists a point $q'$ on $\Gamma_p$ such that $f(q)=f(q')$.
If $f\searrow \lambda_q$ along $\Gamma_q$ with $\lambda_q\ne \lambda$, then
the two sets $\{f(x)\mid x\in\Gamma_p\}$ and $\{f(x)\mid x\in \Gamma_q\}$ should be disjoint
by the property~(iv). 
However $f(q)=f(q')$ is a common element of these two sets. 
If $f\searrow \lambda$ along $\Gamma_q$, then $q\in S_\Omega$.
However, this and $f(q)<f(p)$ contradict $f(p)=\delta=\min_{x\in S_\Omega}f(x)$.
Thus, in either case, a contradiction arises.

Case 2: $Y_{[\lambda,\delta]}\cap f^{-1}(\lambda)\neq \emptyset$. Take one point  $q\in Y_{[\lambda,\delta]}\cap f^{-1}(\lambda)$, then $q$ belongs to the connected component of $f^{-1}([\lambda,\ve])\setminus \text{Int\,}B_{a,R}$ contained in $Y_{[\lambda,\delta]}$ for any $\lambda<\ve<\delta$. This contradicts the fact that $\{Y_t\}$ vanishes at $\infty$ when $t$ tends to $\lambda$.

Next we prove the ``if'' assertion. 
Assume that there exists a local minimum point $p\in \Gamma_{\lambda}\cap\partial B_R$ of $f|_{\partial B_R}$ with $f\searrow \lambda$ along $\Gamma_p$ such that the intersection of the connected component $Z_{f(p)}$ of $f^{-1}(f(p))$ containing $p$ with the sphere $\partial B_R$ consists of isolated points. 
Since $p$ is a local minimum point of $f|_{\partial B_R}$, 
$p$ is also a local minimum point of $r_a|_{Z_{f(a)}}$ by Remark~\ref{rem25}~(3). 
This and the isolatedness of $Z_{f(p)}\cap \partial B_R$ imply that $Z_{f(p)}\subset \R^n\setminus \Int B_R$. 

Put $\delta=f(p)$. 
Let $Z_{(\lambda,\delta]}$ be the connected component of $f^{-1}((\lambda, \delta])$ containing $p$. 
We will show that $Z_{(\lambda,\delta]}\cap \partial B_R= Z_\delta\cap \partial B_R$. 
It is easy to see that 
the 
two connected components $\Gamma_p\setminus\{p\}$ and $\partial B_R\setminus Z_\delta$ are subsets of different connected components of $\R\setminus Z_\delta$. Assume that there exists a point $x\in Z_{(\lambda,\delta]}\cap (\partial B_R\setminus Z_\delta)$. Choose a point $y\in \Gamma_p\setminus\{p\}$. Note that the values $f(x)$ and $f(y)$ are in $(\lambda,\delta)$. Since $Z_{(\lambda,\delta]}$ is connected, there exists a path in $Z_{(\lambda,\delta]}$ connecting $x$ and $y$. Furthermore, since $f$ is a trivial fibration on $(\lambda,\delta]$, we can isotope this path so that it is in $Z_{(\lambda,\delta]}\setminus Z_\delta$. However, this is impossible since $x$ and $y$ belong to different connected components of $\R^n\setminus Z_\delta$.
Therefore, $Z_{(\lambda,\delta]}\cap \partial B_R= Z_\delta\cap \partial B_R$. 

Now it follows that for any $t\in (\lambda,\delta)$, the connected component $Z_t$ of $f^{-1}(t)$ intersecting $\Gamma_p$ does not intersect $B_R$. Hence the distance function $r_a|_{Z_t}$ on $Z_t$ attains a minimum value at some point belonging to a tangency branch in $\R^n\setminus \Int B_R$.
Thus, we can find a sequence 
$\{t_k\}$ on $(\lambda,\delta)$ with $\lim_{k\to\infty}t_k=\lambda$ and a point $q\in \Gamma\cap \partial B_R$ such that 
\[
   \min\{r_a(x)\mid x\in Z_{t_k} \}= r_a(q_k),
\]
where $q_k= Z_{t_k}\cap \Gamma_q$. 
The distance $r_a(q_k)$ goes to $\infty$ as $k\to\infty$, otherwise $\Gamma_q$ intersects $f^{-1}(\lambda)$ and this contradicts the property~(i).
Hence $Z_{t_k}$ vanishes at $\infty$ as $k\to\infty$.

The proof for the case where $p$ is a local maximum point is similar.
\end{proof}

\begin{rem}\label{rem32}
Using Theorem~\ref{thm31}, a vanishing component at $\infty$ of $f:\R^n\to\R$ is detected as follows:
\begin{itemize}
\item[(Step 1)] Choose a generic center $a$, calculate all tangency branches, and fix a sufficiently large radius $R>0$ that satisfies the conditions written in Section~\ref{sec24}.
\item[(Step 2)] For each $p\in (\Gamma\setminus \Sing(f))\cap \partial B_R$, calculate $\lambda_p=\lim_{r\to\infty}f(x_p(r))$, where $x_p(r)=\Gamma_p\cap \partial B_r$. Then, make the following lists of finite sets:
\[
\begin{split}
   P_{\rm min}(\lambda)&=\{p\in (\Gamma\setminus \Sing(f))\cap \partial B_R\mid \text{$p$ is local minimum of $f|_{\partial B_R}$ with $f\searrow \lambda$}\} \\
   P_{\rm max}(\lambda)&=\{p\in (\Gamma\setminus \Sing(f))\cap \partial B_R\mid \text{$p$ is local maximum of $f|_{\partial B_R}$ with $f\nearrow \lambda$}\} \\
   \Lambda_{\rm min}&=\{\lambda\in\R\mid P_{\rm min}(\lambda)\ne\emptyset\} \\
   \Lambda_{\rm max}&=\{\lambda\in\R\mid P_{\rm max}(\lambda)\ne\emptyset\}.
\end{split}
\]
\item[(Step 3)] 
For each element $\lambda\in \Lambda_{\rm min}$ (resp. $\lambda\in\Lambda_{\rm max}$), check if there exists $p\in P_{\rm min}(\lambda)$ (resp. $p\in P_{\rm max}(\lambda)$) such that the intersection $f^{-1}(f(p))\cap \partial B_R$ consists of isolated points.
\begin{itemize}
\item [(3-1)]If it exists, then there exists a vanishing component at $\infty$ when $t$ tends to $\lambda$ by Theorem~\ref{thm31}.
\item[(3-2)] If it does not exist, then $\dim f^{-1}(f(p))\cap \partial B_R\geq 1$ for any $p\in P_{\rm min}(\lambda)\cup P_{\rm max}(\lambda)$.
For each $p\in P_{\rm min}(\lambda)\cup P_{\rm max}(\lambda)$, 
calculate all critical values $c_1,\ldots,c_k$ of $r_a(x)=\|x-a\|$ on $f^{-1}(f(p))$ and then choose a real number $R'$ greater than $\max\{R, c_1,\ldots,c_k\}$, see Figure~\ref{fig5}. 
Make a list $L'$ of the connected components of $\partial B_{R'} \setminus f^{-1}(f(p))$
and find a component $\Omega'_p\in L'$ intersecting $\Gamma_p$.
If $f^{-1}(\lambda)\cap \Omega_p'=\emptyset$ then there exists a vanishing component at $\infty$ when $t$ tends to $\lambda$ as shown in the next lemma (Lemma~\ref{lemma32}).
\end{itemize}
\end{itemize}
All vanishing components at $\infty$ are detected by the above steps, which is proved in Lemma~\ref{lemma34} below.
\end{rem}

\begin{figure}[htbp]
\includegraphics[scale=1.16, bb=187 585 372 710]{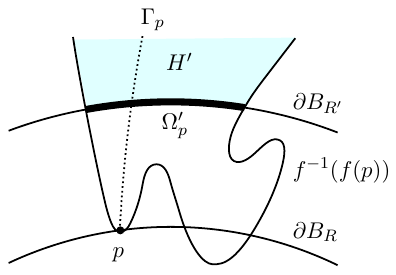}
\caption{A schematic picture for Step (3-2).}
\label{fig5}
\end{figure}

\begin{lemma}\label{lemma32}
If $f^{-1}(\lambda)\cap \Omega_p'=\emptyset$ then there exists a vanishing component at $\infty$ when $t$ tends to $\lambda$.
\end{lemma}

\begin{proof}
Consider the case where $p\in P_{\rm min}(\lambda)$.
Let $H'$ be the connected component of $\R^n\setminus (f^{-1}(f(p))\cup \Int B_{R'})$ intersecting $\Gamma_p$ and $\bar H'$ be its closure.
Let $\bar\Omega_p'$ be the closure of $\Omega_p'$ in $\partial B_{R'}$. 
Since $R'>\max\{R, c_1,\ldots,c_k\}$, $\bar H'$ is diffeomorphic to $\bar\Omega_p'\times [0,1)$.
Let $Y_{(\lambda,f(p)]}$ be the connected component of $f^{-1}((\lambda,f(p)])\setminus\Int B_{R'}$ intersecting $\Gamma_p$ and $\bar Y_{(\lambda,f(p)]}$ be its closure. 
The inclusion $\bar Y_{(\lambda,f(p)]}\subset \bar H'$, the property~(iii), and the assumption
$f^{-1}(\lambda)\cap \Omega_p'=\emptyset$ imply that
 $\bar Y_{(\lambda,f(p)]}\cap f^{-1}(\lambda)=\emptyset$.
Since $f\searrow\lambda$ along $\Gamma_p$ and 
$f^{-1}(\lambda)\cap \Omega_p'=\emptyset$,  
we have $\lambda<f(x)$ for $x\in \bar\Omega_p'$, $f(x)=f(p)$ for $x\in\partial\bar\Omega_p'$,
and there exists a point $x'\in\Omega_p'$ such that $f(x')<f(p)$.
Set $\delta=\min\{f(x)\mid x\in \bar\Omega_p'\}$. Note that $\lambda<\delta<f(p)$.
For $t\in (\lambda, \delta)$, the connected component $Y_t$ of $f^{-1}(t)$ intersecting $\Gamma_p$ does not intersect $\Omega_p'$,  and therefore it is contained in $Y_{(\lambda,f(p)]}$.
Since $\bar Y_{(\lambda,f(p)]}\cap f^{-1}(\lambda)=\emptyset$, 
$\{Y_t\}$ is a vanishing component at $\infty$ when $t$ tends to $\lambda$.

The assertion for the case $p\in P_{\rm max}(\lambda)$ is proved similarly.
\end{proof}

\begin{lemma}\label{lemma34}
If there exists a vanishing component at $\infty$ when $t$ tends to $\lambda$, then
there exists a point $p\in (\Gamma\setminus(\Sing(f)))\cap \partial B_R$ with  $f^{-1}(\lambda)\cap\Omega_p'=\emptyset$.
\end{lemma}

\begin{proof}
Suppose that there exists a vanishing component at $\infty$ when $t$ tends to $\lambda$. We prove only the case $t>\lambda$.
By Theorem~\ref{thm31}, there exists a local minimum point $p\in\Gamma_p\cap \partial B_R$ of $f|_{\partial B_R}$ such that the intersection of the connected component $Z_{f(p)}$ of $f^{-1}(f(p))$ containing $p$ with the sphere $\partial B_R$ consists of isolated points. 
Put $\delta=f(p)$ and let $Z_{(\lambda,\delta]}$ be the connected component of $f^{-1}((\lambda,\delta])$ containing $p$.
Then, as shown in the proof of the ``if'' assertion of Theorem~\ref{thm31}, we have $Z_{(\lambda,\delta]}\cap \partial B_R=Z_\delta\cap \partial B_R$. 
Let $R'$ be the radius chosen as in~(3-2) and $\Omega_p'$ be the connected component of $\partial B_{R'}\setminus Z_\delta$ intersecting $\Gamma_p$. Assume that there exists an intersection point $x\in f^{-1}(\lambda) \cap \Omega_p'$.
By the property~(iii), there exists an arc on $f^{-1}(\lambda)$ connecting $x$ and a point on $f^{-1}(\lambda)\cap \partial B_R$, but such an arc should intersect $Z_\delta$. This contradicts the fact that the image of this arc is $\lambda$.
\end{proof}

\section{Proof of Theorem~\ref{thm1}}\label{sec4}

Now, we restrict our setting to the case of polynomial functions with three variables.
Let $f:\mathbb{R}^3\to \mathbb{R}$ be a polynomial function.
For each $\lambda\in T_{\infty}(f)\setminus K_0(f)$, 
there exists a sufficiently small $\ve>0$ such that, for $I_\lambda^-=(\lambda-\ve,\lambda)$ and 
$I_\lambda^+=(\lambda,\lambda+\ve)$, 
the restriction of $f$ to $f^{-1}(I_\lambda^*)$ and the restriction of $f$ to $f^{-1}(I_\lambda^*)\cap B_R$ 
are trivial fibrations unless $f^{-1}(I_\lambda^*)=\emptyset$, where $*\in\{-, +\}$.
Here $\ve$ is chosen so that $f^{-1}(t)$ intersects $\partial B_R$ transversely for $t\in I_\lambda^*$.
Then the restriction of $f$ to $f^{-1}(I_\lambda^*)\cap (\R^n\setminus \text{Int} B_R)$ is also a trivial fibration. 

The surface $f^{-1}(\lambda)\setminus \text{Int} B_R$ divides $\R^3\setminus \text{Int} B_R$ into a finite number of
connected components $H^*_{\lambda,1},\ldots,H^*_{\lambda,n_\lambda}$ by the property~(iii),
where $*=-$ if $f(x)<\lambda$ on $H^*_{\lambda,i}$ and $*=+$ if $f(x)>\lambda$ on $H^*_{\lambda,i}$.
Each $H^*_{\lambda,i}$ is homeomorphic to $\Omega^*_{\lambda,i}\times [0, 1)$, where $\Omega^*_{\lambda,i}=H^*_{\lambda,i}\cap\partial B_R$. 

Let $\Ind(\lambda, \Omega_{\lambda,i}^*)$ be the sum of indices of the gradient vector field of $f$ restricted to $\partial B_R$ for all zeros belonging to $\Gamma^{(\lambda)}$ on $\Omega_{\lambda,i}^*$ as defined in the introduction.

\begin{lemma}\label{lmindex}
$\chi(f^{-1}(t)\cap H^*_{\lambda,i})=\Ind(\lambda, \Omega^*_{\lambda,i})$ for any $t\in I_\lambda^*$.
\end{lemma}


\begin{proof} 
Choose $R'>R$ sufficiently large so that $f^{-1}(t)$ intersects $\partial B_{R'}$ transversely and $f^{-1}(t)\setminus \Int B_{R'}$ is diffeomorphic to $(f^{-1}(t)\cap \partial B_R) \times [0,1)$ for $t\in I_\lambda^*$. 
Then $f^{-1}(t)\cap H_{\lambda,i}^*$ has the same homotopy type as $f^{-1}(t)\cap H_{\lambda,i}^*\cap B_R^{R'}$, where $B_R^{R'}=\{x\in\R^3\mid R\leq \|x-a\|\leq R'\}$.
Hence we have 
\[
   \chi(f^{-1}(t)\cap H^*_{\lambda, i})= \chi(f^{-1}(t)\cap H_{\lambda,i}^*\cap B_R^{R'}).
\]

Consider the distance function $r_a(x)=\|x-a\|$ on $f^{-1}(t)\cap H^*_{\lambda,i}\cap B_R^{R'}$.
Due to a generic choice of the center of $B_R$ in Section~\ref{sec22}, 
this function has only non-degenerate critical points and has no critical point on the boundary. 
Hence, there is a one-to-one correspondence between critical points of $r_a$ on $f^{-1}(t)\cap H^*_{\lambda,i}$ and the tangency branches $\Gamma_p$ passing through $p\in \Gamma^{(\lambda)}\cap \Omega_{\lambda,i}^*$ as mentioned in Remark~\ref{rem25}~(1).
If $p\in \Gamma^{(\lambda)}\cap \Omega_{\lambda,i}^*$ is local minimum or maximum of $f|_{\partial B_R}$ 
then $\Ind_p(X_{a,R})=1$ and the Morse index $i(p_t)$ of the distance function $r_a$ on $f^{-1}(t)\cap H^*_{\lambda,i}$ at the intersection point $p_t$ of $\Gamma_p$ with $f^{-1}(t)$ is $0$ or $2$.
If $p\in \Gamma^{(\lambda)}\cap \Omega_{\lambda,i}$ is a saddle point of $f|_{\partial B_R}$ 
then $\Ind_p(X_{a,R})=-1$ and the Morse index $i(p_t)$ at the intersection point $p_t$ of $\Gamma_p$ with $f^{-1}(t)$ is $1$. 
Hence we have $\Ind_p(X_{a,R})=(-1)^{i(p_t)}$.
Since the Euler characteristic of $f^{-1}(t)\cap \Omega^*_{\lambda,i}$ is $0$, by the Morse Theory, we have
\[
\begin{split}
   \chi(f^{-1}(t)\cap H_{\lambda,i}^*\cap B_R^{R'})&=\sum_{p\in \Gamma^{(\lambda)}\cap \Omega_{\lambda,i}}(-1)^{i(p_t)} \\
   &=\sum_{p\in \Gamma^{(\lambda)}\cap \Omega_{\lambda,i}}\Ind_p(X_{a,R})=\Ind(\lambda,\Omega^*_{\lambda,i}).
\end{split}
\]
This completes the proof. 
\end{proof}

\begin{proof}[Proof of Theorem~\ref{thm1}]
We prove the first assertion by contraposition.
Assume that $\lambda$ is a typical value of $f$.
There exists a sufficiently small $\ve>0$ such that $f$ is a trivial fibration on $I_\ve=(\lambda-\epsilon, \lambda+\epsilon)$. 
Let $\Omega_{\lambda,i}^*$ be a connected component of $\partial B_R\setminus f^{-1}(\lambda)$ and $\partial \bar\Omega_{\lambda,i}^*$ be the boundary of the closure of $\Omega_{\lambda,i}^*$ in $\partial B_R$, which is a union of circles.
By the property (iii), the connected component $Y$ of $f^{-1}(\lambda)\setminus \Int B_R$ intersecting $\partial \bar\Omega_{\lambda,i}^*$ is diffeomorphic to $\partial \bar\Omega_{\lambda,i}^*\times [0,1)$, and hence $\chi(Y)=0$. 
This and the triviality of $f$ on $I_\ve$ imply that $\chi(f^{-1}(t)\cap H_{\lambda,i}^*)=0$ for $t\in I_\lambda^*$, 
where $H_{\lambda,i}^*$ is the connected component of $\R^3\setminus (f^{-1}(\lambda)\cup \Int B_R)$ intersecting $\Omega_{\lambda,i}^*$.
Combining this with Lemma~\ref{lmindex} we obtain $\Ind(\lambda, \Omega_{\lambda,i}^*)=0$.
This completes the proof of the first assertion.

Next we prove the second assertion. 
Because there does not exist a component of $f^{-1}(t)$ vanishing at $\infty$ when $t$ tends to $\lambda$, 
there exists a sufficiently small $\ve>0$ such that
each connected component of $f^{-1}(t)$ intersects $\partial B_R$ for all $t\in I_\lambda^-\cup I_\lambda^+$. 
Let $H_{\lambda,i}^*$ be a connected component of $\R^3\setminus(f^{-1}(\lambda)\cup\Int B_R)$
and $\{Y_t^1,\ldots,Y_t^s\}$ be the connected components of $f^{-1}(t)\cap H_{\lambda,i}^*$.
Set $\Omega_{\lambda,i}^*=H_{\lambda,i}^*\cap \partial B_R$. 
Since $\Ind(\lambda,\Omega_{\lambda,i}^*)=0$, we have $\chi(f^{-1}(t)\cap H_{\lambda,i}^*)=0$ by Lemma~\ref{lmindex}. Hence
\[
   \sum_{j=1}^s\chi(Y_t^j)=0.
\]
Choose $R'>R$ sufficiently large so that $f^{-1}(t)$ intersects $\partial B_{R'}$ transversely and $f^{-1}(t)\setminus \Int B_{R'}$ is diffeomorphic to $(f^{-1}(t)\cap \partial B_{R'})\times [0,1)$ for $t\in I_\lambda^*$, and set $B_R^{R'}=\{x\in\R^3\mid R\leq \|x-a\|\leq R'\}$.
Then, as mentioned at the beginning of the proof of Lemma~\ref{lmindex}, $\chi(Y_t^j \cap B_R^{R'})=\chi(Y_t^j)$ holds.
Hence
\begin{equation}\label{eq41}
   \sum_{j=1}^s\chi(Y_t^j \cap B_R^{R'})=0.
\end{equation}
Here each $Y_t^j \cap B_R^{R'}$ is a compact, connected, orientable surface embedded in $\R^3$.

We claim that $\chi(Y_t^j \cap B_R^{R'})=0$ for any $j=1,\ldots,s$.
If $s=1$ then it follows from equation~\eqref{eq41}.
Suppose that $s\geq 2$. 
Assume that $\chi(Y_t^{j_0} \cap B_R^{R'})\ne 0$ for some $j_0\in\{1,\ldots,s\}$.
Then there exists a connected component $Y_t^{j_1} $ with $\chi(Y_t^{j_1} \cap B_R^{R'})>0$ by~\eqref{eq41}.
Since $Y_t^{j_1} \cap B_R^{R'}$ is a compact, connected, orientable surface, it is diffeomorphic to a disk.
The boundary of this disk lies on $\partial B_R$ since $\ve>0$ is chosen so that $Y_t^{j_1}\cap \partial B_R\ne\emptyset$. Moreover, this boundary is parallel to a boundary component of the closure of $\Omega_{\lambda_i}^*$ due to the property~(iii).
Since $H_{\lambda,i}^*$ is homeomorphic to $\Omega_{\lambda,i}^*\times [0,1)$ and the disk $Y_t^{j_1}\cap B_R^{R'}$ is relatively embedded in $H_{\lambda,i}^*$, $\Omega_{\lambda_i}^*$ should be a disk.
Since the boundary of $\Omega_{\lambda_i}^*$ is connected, $f^{-1}(t)\cap H_{\lambda,i}^*$ is also connected. This contradicts $s\geq 2$.

Now we have $\chi(Y_t^j \cap B_R^{R'})=0$ for any $t\in I_\lambda^*$ and $j=1,\ldots,s$.
This means that all of these connected components are diffeomorphic to $S^1\times [0,1]$. 
Therefore, the relative homotopy groups $\pi_i(f^{-1}(I_\lambda^*\cup\{\lambda\}), f^{-1}(\lambda),x)$ are trivial for all $i\in\N$ and any base point $x\in f^{-1}(\lambda)$. Note that this conclusion holds for both of the cases $*=-$ and $*=+$.
Hence, by~\cite[Proposition 3.3 and Theorem 1.2]{IN}, 
for $I_\lambda=(\lambda-\ve, \lambda+\ve)$, the map
\[
   f|_{f^{-1}(I_\lambda)}: f^{-1}(I_\lambda)\to I_\lambda
\]
is a Serre fibration. Then, this implies that $f|_{f^{-1}(I_\lambda)}$ is a trivial fibration by~\cite[Corollary 32]{Mei02}. Hence $\lambda$ is a typical value at $\infty$ of $f$. 
\end{proof}

\section{Typical values of polynomial maps with $2$-dimensional fibers}\label{sec5}

In this section, we study polynomial maps from $\R^n$ to $\R^{n-2}$. The case $n=3$ is studied in the previous section.

Let $F:\R^n\to \R^{n-2}$ be a polynomial map, where $n\geq 3$, and $\lambda$ be a point in $F(\R^n)\setminus \bar K_0(F)$,
where $\bar K_0(F)$ is the closure of $K_0(F)$ in $\R^{n-2}$.
Let $B_R$ be the $n$-dimensional ball in $\R^n$ centered at $a\in\R^n$ and with radius $R>0$.
As shown in \cite[Lemma 3.2]{IN}, we can choose a sufficiently large radius $R>0$ satisfying the following property:
\begin{itemize}
\item[(v)]  Each connected component $Y$ of $F^{-1}(\lambda)\setminus \Int B_R$ 
intersects $\partial B_r$ transversely for any $r\geq R$. In particular, 
$Y\setminus \Int B_r$ is diffeomorphic to $(Y\cap \partial B_r)\times [0,1)$ for any $r\geq R$.
\end{itemize}
In particular, there is a deformation-retract from $F^{-1}(\lambda)$ to $F^{-1}(\lambda)\cap B_R$.

\begin{thm}\label{thm51} 
Let $F:\R^n\to \R^{n-2}$ be a polynomial map, where $n\geq 3$.
For $\lambda\in F(\R^n)\setminus \bar K_0(F)$, choose a radius $R$ so that the property~{\rm (v)} holds.
Then, $\lambda$ is a typical value at $\infty$ of $F$ if and only if the following are satisfied:
\begin{itemize}
\item[(1)] There is no vanishing component at $\infty$ when $t$ tends to $\lambda$;
\item[(2)] There exists a neighborhood $D$ of $\lambda$ in $\R^{n-2}$ such that,  for all $t\in D$, 
\begin{itemize}
\item[(2-1)] $F^{-1}(t)\setminus \Int B_R$ has no compact, connected component, and
\item[(2-2)] $\chi(F^{-1}(t))= \chi(F^{-1}(\lambda))$ holds.
\end{itemize}
\end{itemize}
\end{thm}

\begin{proof}
It is enough to show that if the conditions (1) and (2) are satisfied then $F$ is a trivial fibration over some neighborhood of $\lambda$. Assume that the two conditions are satisfied. Let $D$ be a small neighborhood of $\lambda$ as in the condition~(2).
We can choose $D$ small enough so that the fibers $F^{-1}(t)$ are regular and intersect $\partial B_R$ transversely for all $t\in D$. The map $F|_{F^{-1}(D)\cap B_R}: F^{-1}(D)\cap B_R\to D$ is a trivial fibration. 

By the conditions~(1) and~(2-1), $F^{-1}(t)\setminus \Int B_R$ does not have a connected component which is contractible for any $t\in D$. Hence we have $\chi(F^{-1}(t)\setminus \Int B_R)\leq 0$. Then, by the condition~(2-2) and the property (v), we have 
\[
\begin{split}
   \chi(F^{-1}(\lambda))&=\chi(F^{-1}(t))=\chi(F^{-1}(t)\cap B_R)+\chi(F^{-1}(t)\setminus\Int B_R) \\
&\leq  \chi(F^{-1}(t)\cap B_R)=\chi(F^{-1}(\lambda)\cap B_R)=\chi(F^{-1}(\lambda)),
\end{split}
\]
which implies that $\chi(F^{-1}(t)\setminus\Int B_R)=0$.
Here we used the fact that $F^{-1}(t)\cap \partial B_R$ is a disjoint union of circles and its Euler characteristic is $0$.
By the condition (2-1), $F^{-1}(t)\setminus\Int B_R$ is diffeomorphic to a disjoint union of a finite number of copies of $S^1\times [0,1)$. 

The rest of the proof is same as the last argument in the proof of Theorem~\ref{thm1}.
Since there exists a deformation-retract from $F^{-1}(t)$ to $F^{-1}(t)\cap B_R$ for each $t\in D$
and the map $F|_{F^{-1}(D)\cap B_R}$ is a trivial fibration,
the relative homotopy groups $\pi_i(F^{-1}(D), F^{-1}(\lambda),x)$ are trivial for all $i\in\N$ and any base point $x\in f^{-1}(\lambda)$. 
Then, by~\cite[Proposition 3.3 and Theorem 1.2]{IN}, the map $F|_{F^{-1}(D)}: F^{-1}(D)\to D$ is a Serre fibration and hence it is a trivial fibration by~\cite[Corollary 32]{Mei02}. Hence $\lambda$ is a typical value at $\infty$ of $F$. 
\end{proof}


\begin{thebibliography}{99}

\bibitem{BCR}
J.~Bochnak, M.~Coste, M-F. Roy,
Real algebraic geometry, 
Ergeb. Math. Grenzgeb. (3), 36, Springer-Verlag, Berlin, 1998.

\bibitem{cp}
M. Coste, M.J. de la Puente, 
{\it Atypical values at infinity of a polynomial function on the real plane: an erratum, and an algorithmic criterion},
J. Pure Appl. Algebra {\bf 162} (2001), no. 1, 23--35. 

\bibitem{DJT21}
L.R.G. Dias, C. Joi\c{t}a and M. Tib\u{a}r,
{\it Atypical points at infinity and algorithmic detection of the bifurcation locus of real polynomials},
Math. Z. {\bf 298} (2021), no.3-4, 1545--1558.

\bibitem{HP}
H.V. Ha and T.S. Pham, 
Genericity in polynomial optimization
Ser. Optim. Appl., 3
World Scientific Publishing Co. Pte. Ltd., Hackensack, NJ, xix+240 pp, 2017.

\bibitem{IN}
M. Ishikawa, T. T. Nguyen, {\it Relative homotopy groups and Serre fibrations for polynomial maps}, arXiv:2208.10055, to appear in J. Math. Soc. Japan.


\bibitem{INP19}
M.~Ishikawa, T.T. Nguyen and T.S. Pham,
{\it Bifurcation sets of real polynomial functions of two variables and Newton polygons},
J. Math. Soc. Japan {\bf 71} (2019), no. 4, 1201--1222.

\bibitem{JT17}
C. Joi\c{t}a and M. Tib\u{a}r,
{\it Bifurcation values of families of real curves}, 
Proc. Roy. Soc. Edinburgh Sect. A {\bf 147} (2017), no. 6, 1233--1242. 

\bibitem{JT18}
C. Joi\c{t}a and M. Tib\u{a}r,
{\it Bifurcation set of multi-parameter families of complex curves},
J. Topol. {\bf 11} (2018), 739--751.

\bibitem{KOS00}
K.~Kurdyka, P.~Orro, S.~Simon, 
{\it Semialgebraic Sard theorem for generalized critical values},
J. Differential Geom. {\bf 56} (2000), no.1, 67--92.

\bibitem{Mei02}
G.~Meigniez,
{\it Submersions, fibrations and bundles},
Trans. Amer. Math. Soc. {\bf 354} (2002), 3771--3787.

\bibitem{Mil63}
J. Milnor, Morse Theory. Based on lecture notes by M. Spivak and R. Wells, 
Ann. of Math. Stud., 51, Princeton University Press, Princeton, NJ, 1963.

\bibitem{Tho69}
R. Thom, 
{\it Ensembles et morphismes stratifi\'{e}s}, 
Bull. Amer. Math. Soc. {\bf 75} (1969), 240--284.

\bibitem{tz}
M. Tib\u{a}r, A. Zaharia,
{\it Asymptotic behaviour of families of real curves},
Manuscripta Math. {\bf 99} (1999), 383--393.

\bibitem{Ver76}
J.-L.~Verdier, 
{\it Stratifications de Whitney et th\'{e}or\`{e}me de Bertini-Sard},
Invent. Math. {\bf 36} (1976), 295--312


\end{thebibliography}
\end{document}